\documentclass[a4paper,numbook,babel,final,envcountsame,11pt]{article}

\usepackage{amssymb}
\usepackage{amsthm}
\usepackage{amsmath}
\usepackage{graphicx}
\usepackage{array,hhline}
\numberwithin{figure}{section}

\newcommand{\Aset}{\mathcal{A}}
\newcommand{\Bset}{\mathcal{B}}
\newcommand{\Cset}{\mathcal{C}}

\newcommand{\Xset}{\mathcal{X}}
\newcommand{\Yset}{\mathcal{Y}}

\newtheorem{thm}{Theorem}[section]

\newtheorem*{cor*}{Corollary}
\newtheorem{defn}[thm]{Definition}

\newtheorem{proposition}[thm]{Proposition}
\newtheorem{lem}[thm]{Lemma}
\begin{document}

     % Enter your date or \today between curly braces
\title{Sets of lines passing through one point each of three sets
of points}

\title{\Large{\bf On a result concerning \ algebraic curves \ passing\\  through\ $n$-independent \ nodes}}
\author{H. A. Hakopian\\ \\
\emph{Department of Informatics and Applied Mathematics,}\\  \emph{Yerevan State University,}\\ \emph{Institute of Mathematics of NAS RA}}

\date{}

\maketitle
%%%%%%%%%%%%%%%%%%%%%%%%%%%%%%%%%%%%%%%%%%%%%%%%%%%%%%%%%%%%%%%%%%%%%%%%
\begin{abstract}
 Let a set of nodes $\mathcal X$ in the plane be $n$-independent, i.e., each node has a fundamental polynomial of degree $n.$ Assume that\\
$\#\mathcal X=d(n,n-3)+3= (n+1)+n+\cdots+5+3.$ In this paper we prove that there are at most three linearly independent curves of degree less than or equal to $n-1$ that pass through all the nodes of $\mathcal X.$ We provide a characterization of the case when there are exactly three such curves. Namely, we prove that then the set $\mathcal X$ has a very special construction: either all its nodes belong to a curve of degree $n-2,$ or  all its nodes but three belong to a (maximal) curve of degree $n-3.$

This result complements a result established recently by H. Kloyan, D. Voskanyan, and H. H. Note that the proofs of the two results are completely different.

\end{abstract}

\parindent=1cm

\section{Introduction}

 Denote the space of all bivariate polynomials of total degree $\le n$ by $$\Pi_n= \left\{\sum_{i+j\leq n} a_{ij} x^i y^j \right\}.$$
We have that
%\begin{equation*}
$ N := N_n := \dim \Pi_n = (1/2)(n+1)(n+2).$
%\end{equation*}

Denote by $\Pi$ the space of all bivariate polynomials.

A \emph{plane algebraic curve} is the zero set of some bivariate polynomial of degree $\ge 1.$~To simplify notation, we shall use the same letter,  say $p$,
to denote the polynomial $p\in\Pi$ and the curve given by the equation $p(x,y)=0$.
In particular, by $\ell$ we denote a linear
polynomial from $\Pi_1$ and the line defined by the equation
$\ell(x, y)=0.$

Consider a set of $s$ distinct nodes ${\mathcal X}= {\mathcal X}_s = \{(x_1,y_1), (x_2,y_2), \ldots, (x_s,y_s) \}.$
 The problem of finding a polynomial $p \in \Pi_n,$ which satisfies the conditions
\begin{equation}\label{eq:intpr}
  p(x_i,y_i) = c_i, \quad i=1, \ldots, s,
\end{equation}
is called interpolation problem.

Denote by $p\big\vert_{\mathcal X}$ the restriction of $p\in\Pi$ on $\mathcal X.$

A polynomial $p \in \Pi_n$ is called a fundamental polynomial for a node \linebreak$A\in{\mathcal X}$ if
$p(A)=1\ \ \text{and}\ \ p\big\vert_{{\mathcal X}\setminus\{A\}} = 0.$ 

We denote this $n$-fundamental polynomial
by $p_{A}^\star:=p_{A,\Xset}^\star.$

\begin{defn} \label{poised}
The interpolation problem with a set of nodes ${\mathcal X}_s$ is called $n$-poised if for any data $(c_1,\ldots, c_s)$
there is a unique polynomial $p\in\Pi_n$ satisfying the interpolation conditions \eqref{eq:intpr}.
\end{defn}
A necessary condition of
poisedness is $\#{\mathcal X}_s=s = N.$

Now, let us consider the concept of $n$-independence (see \cite{E,HJZ}).

\begin{defn}
A set of nodes ${\mathcal X}_s$ is called $n$-independent, if all its
nodes have $n$-fundamental polynomials. Otherwise, it is called $n$-dependent.
\end{defn}

Fundamental polynomials are linearly independent.~Therefore a necessary condition of $n$-independence for ${\mathcal X}_s$
is $s \leq N$.

In this paper we consider $n$-independence more generally. Namely, we admit possibility
to include  in the $n$-independent set ${\mathcal X}_s$ a \emph{directional derivative node}, denoted by $A^{(k)}.$ We have that $p(A^{(k)}):=D_{\bf a}^kp(A),$ where $p\in\Pi,\ {\bf a}$  is a direction, and $k\in\mathbb{N}.$ For a node $A^{(k)}$  we assume in addition that
$$p\in\Pi_n,\ p|_\Xset=0 \implies  D_{\bf a}^ip(A)=0,\ i=0,\ldots,k-1.$$
The set $\Xset\cup\{A^{(k)}\}$ is $n$-independent means that $\Xset$ is $n$-independent and the node $A^{(k)}$ has an $n$-fundamental polynomial $p=p^\star_{A^{(k)}}:$
$$p\in\Pi_n,\ p|_\Xset=0,\ D_{\bf a}^kp(A)=1.$$

We say that a node $A^{(k)}$ belongs to a curve $q$ if $D_{\bf a}^ip(A)=0,\ i=0,\ldots,k.$ In particular $A^{(k)}$ belongs to a line $\ell$ if $A\in \ell$ and ${\bf a}$ is the direction vector of $\ell.$

Let us mention, as it can be readily verified, that  all the results we present below concerning  $n$-independent sets hold true for the above mentioned generalization.

\subsection{Some properties of $n$-independent nodes}

 Let us start with the following
\begin{lem}[Lemma 2.2, \cite{HM}] \label{XA}
Suppose that a set of nodes $\Xset$ is $n$-independent and a node $A\notin\Xset$ has
an $n$-fundamental polynomial with respect to the set $\Xset\cup\{A\}.$ Then the latter set  is $n$-independent too.
\end{lem}

Denote the distance between the points $A$ and $B$ by $\rho(A,B).$ Let us recall the following (see Rem. 1.14, \cite{H})
\begin{lem} \label{eps'}
Suppose that ${\mathcal X}_s=\{A_i\}_{i=1}^s$ is an $n$-independent set. Then there is a number $\epsilon>0$ such that
any set ${\mathcal X}_s'=\{A_i'\}_{i=1}^s,$ with the property that  $\rho(A_i,A_i')<\epsilon,\ i=1,\ldots,s,$  is $n$-independent too.
\end{lem}

Next result concerns the extensions of $n$-independent sets.
\begin{lem}[Lemma 2.1, \cite{HJZ}]\label{ext}
Any $n$-independent set ${\mathcal X}$ with $\#{\mathcal X}<N$ can be enlarged to an $n$-poised set.
\end{lem}

Denote the linear space of polynomials of total degree at most $n$
vanishing on ${\mathcal X}$ by
\begin{equation*}{{\mathcal P}}_{n,{\mathcal X}}=\left\{p\in \Pi_n
: p\big\vert_{\mathcal X}=0\right\}.
\end{equation*}
The following two propositions are well-known  (see, e.g., \cite{HJZ}).
\begin{proposition} \label{PnX}
\textit {For any node set ${\mathcal X}$ we have that
\begin{equation*}\label{eq:theta1} \dim {{\mathcal P}}_{n,{\mathcal X}} = N - \#{\mathcal Y},\end{equation*}
where ${\mathcal Y}$ is a maximal $n$-independent subset of ${\mathcal X}.$}
\end{proposition}
\begin{proposition}\label{maxline}
If a polynomial $p\in \Pi_n$ vanishes at $n+1$ points of  a line $\ell$, then we have that
$p = \ell r,$ where $r \in \Pi_{n-1}.$
\end{proposition}

In the sequel we will need the following
\begin{proposition}[Prop. 1.10, \cite{HM}]\label{hm}
Let $\Xset$ be a set of nodes. Then the following two conditions are equivalent:

i)  ${{\mathcal P}}_{n,{\mathcal X}}=\{0\};$

ii) The node set $\Xset$ has an $n$-poised subset.
\end{proposition}

Set $d(n, k) := N_n - N_{n-k} = (1/2) k(2n+3-k).$
The following is a  generalization of Proposition \ref{maxline}.
\begin{proposition}[Prop. 3.1, \cite{Raf}]\label{maxcurve}
Let $q$ be an algebraic curve of degree $k \le n$ without multiple components. Then the following hold:

$i)$ any subset of $q$ containing more than $d(n,k)$ nodes is
$n$-dependent;

$ii)$ any subset ${\mathcal X}$ of $q$ containing exactly $d=d(n,k)$ nodes is $n$-independent if and only if the following condition
holds:
\begin{equation}\label{p=qr}
p\in {\Pi_{n}}\quad \text{and} \quad p|_{\mathcal X} = 0 \Longrightarrow  p = qr,\ \hbox{where}\ r \in \Pi_{n-k}.
\end{equation}
\end{proposition}
Thus, according to Proposition~\ref{maxcurve}, $i)$, at most $d(n,k)$ nodes of $\mathcal X$ can lie in a curve $q$ of degree $k \le n$.
This motivates the following
\begin{defn}[Def. 3.1, \cite{Raf}]\label{def:maximal}
Given an $n$-independent set of nodes $\mathcal X_s$ with $s\ge d(n,k).$ A curve of degree $k \le n$ passing through $d(n,k)$ points
of $\mathcal X_s$ is called maximal.
\end{defn}

We say that a node $A$ of an $n$-poised set ${\mathcal X}$ \emph{uses} a curve $q\in\Pi_k,$ if the latter divides the $n$-fundamental polynomial of $A,$ i.e.,
$p^\star_{A} = q r,\ r \in \Pi_{n-k}.$

Let us bring a characterization of maximal curves:

\begin{proposition}[Prop. 3.3, \cite{Raf}] \label{maxcor}
Let a node set ${\mathcal X}$ be $n$-independent. Then
a curve $\mu$ of degree $k,\ k\le n,$ is a maximal curve if and only if
$$p\in\Pi_n, \ p|_{\Xset\cap\mu}=0 \implies p=\mu q, \ q\in\Pi_{n-k}.$$
\end{proposition}

Next result concerns  maximal independent sets in curves.
\begin{proposition}[Prop. 3.5, \cite{HakTor}] \label{extcurve}
Assume that $\sigma$ is an algebraic curve  of degree $k$ without multiple components and ${\mathcal X}_s\subset \sigma$ is any $n$-independent node set of cardinality $s,\ s<d(n,k).$ Then the set ${\mathcal X}_s$ can be extended to a maximal $n$-independent set ${\mathcal X}_{d}\subset \sigma$ of cardinality $d=d(n,k)$.
\end{proposition}

Next result from Algebraic Geometry will be used in the sequel:
\begin{thm}[Th.~2.2, \cite{W}] \label{intn} If $\Cset$ is a curve of degree $n$ with no multiple components, then through any point $O$ not in $\Cset$ there pass lines which intersect $\Cset$ in $n$ distinct points.
\end{thm}
\noindent Let us mention that, as it follows from the proof, if a line $\ell$ through the point $O$ intersects $\Cset$ in $n$ distinct points then any line through $O,$ sufficiently close to $\ell,$ has the same property.

Finally, let us present a well-known
\begin{lem} \label{2cor}
Given $m$ linearly independent polynomials, $m\ge 2.$
Then for any point $A$ there are $m-1$ linearly independent polynomials, in their linear span, vanishing at $A.$
\end{lem}

\section {A result and its complement}
In this paper we complement the following
\begin{thm}[Thm. 2.5, \cite{HKV}]\label{hkv}
Assume that ${\mathcal X}$ is an $n$-independent set of $d(n, k-2)+3$ nodes with $3\le k\le n-2.$ Then at most three linearly independent curves of degree
$\le k$ may pass through all the nodes of ${\mathcal X}.$ Moreover, there are such three curves for the set ${\mathcal X}$ if and only if all the nodes of ${\mathcal X}$  lie in a  curve of degree $k-1,$ or all the nodes of  ${\mathcal X}$ but three lie in a (maximal) curve of degree $k-2.$
\end{thm}
\noindent Namely, we prove that the above result is true also in the case $k=n-1:$ 
\begin{proposition}\label{hh}
Assume that ${\mathcal X}$ is an $n$-independent set of $d(n, n-3)+3$ nodes, $n\ge 4.$ Then at most three linearly independent curves of degree
$\le n-1$ may pass through all the nodes of ${\mathcal X}.$ Moreover, there are such three curves for the set ${\mathcal X}$ if and only if all the nodes of ${\mathcal X}$  lie in a  curve of degree $n-2,$ or all the nodes of  ${\mathcal X}$ but three lie in a (maximal) curve of degree $n-3.$
\end{proposition}
In the sequel we will use the following
\begin{thm}[Th.~3,  \cite{HK}]\label{hk}
Assume that ${\mathcal X}$ is an $n$-independent set of $d(n, k-2)+2$ nodes with $3\le k\le n-1.$ Then at most four linearly independent curves of degree
$\le k$ may pass through all the nodes of ${\mathcal X}.$ Moreover, there are such four curves for the set ${\mathcal X}$ if and only if all the nodes of ${\mathcal X}$ but two lie in a maximal curve of degree $k-2.$
\end{thm}

\section{Proof of Proposition \ref{hh} }

Assume by way of contradiction that there are four linearly independent curves of degree
$\le n-1$ passing through all the nodes of the $n$-independent set $\Xset,$ with $\#{\mathcal X} = d(n, n-3)+3.$ Then, according to Theorem \ref{hk}, all the nodes of $\Xset$ but three belong to a maximal curve $\mu$ of degree $n-3.$
The curve $\mu$ is maximal and the remaining three nodes of ${\mathcal X},$ denoted by $A, B$ and $C,$ are outside of it: $A,B,C\notin \mu.$ Hence we have that
$${\mathcal P}_{n-1,{\mathcal X}}=\left\{p\in \Pi_{n-1}: p_{\mathcal X}=0\right\}= \left\{{q\mu : q\in \Pi_{2}},\ q(A)=q(B)=q(C)=0\right\}. $$
Thus we get readily that

\noindent $\quad \dim{\mathcal P}_{n-1,{\mathcal X}}=\dim\left\{{q\in \Pi_{2}}: q(A)=q(B)=q(C)=0\right\}$ $=\dim{\mathcal P}_{2, \{A,B,C\}}\\ =6-3=3,$
which contradicts our assumption.
Note that in the last equality we use Proposition \ref{PnX} and the fact that any three nodes are $2$-independent.

Now, let us verify the part ``if''. By assuming that there is a curve $\sigma$ of degree $n-3$ passing through the nodes of $\Xset$ we find readily three linearly independent curves of degree $\le n-1:\ $ $\sigma, x\sigma,
y\sigma,$ passing through $\Xset.$  While if we assume that all the nodes of  ${\mathcal X}$ but three lie in a curve $\mu$ of degree $n-3$ then above evaluation shows that $\dim{\mathcal P}_{n-1,{\mathcal X}}=3.$

\noindent Note that till here the proof was similar to the proof of Theorem \ref{hkv} in \cite{HKV}.

Finally, let us verify the part ``only if''.
Denote the three  curves passing through all the nodes of the set ${\mathcal X}$  by $\sigma_1,\sigma_2,\sigma_3.$
If one of them is of degree $n-2$ then the conclusion of Theorem is satisfied and we are done. Thus, we may assume that each curve is of exact degree $n-1$ and has no multiple components.

We start with two  nodes  $B_1,B_2\notin{\mathcal X}$ for which the following conditions are satisfied, where  the line between $B_1$ and $B_2$ is denoted by $\ell_{12}.$

\noindent $i)$ The nodes $B_1,B_2$ do not belong to the curves $\sigma_1,\sigma_2,\sigma_3$;

\noindent $ii)$ The set ${\mathcal X}\cup \{B_1,B_2\}$ is $n$-independent;

\noindent $iii)$ The line $\ell_{12}$ does not pass through any node from ${\mathcal X};$

\noindent $iv)$ The line $\ell_{12}$ intersects each of the curves $\sigma_1,\sigma_2,\sigma_3,$ at $n-1$ different points. Moreover, it intersects any two different components of these curves at different points.

Let us verify that one can find such two nodes. Indeed, in view of Lemma \ref{ext}, we can start by choosing some nodes $B_i', i=1,2,$ satisfying the conditions
$i)$ and $ii)$.  Then, according to Lemma \ref{eps'}, for some positive $\epsilon$ any two nodes
in the $\epsilon$ neighborhoods of $B_i', i=1,2,$ respectively, satisfy the first two conditions.

Next, from these  neighborhoods, in view of Theorem \ref{intn},  we can choose the nodes $B_i,\ i=1,2,$
satisfying the condition $iii)$ and $iv)$ too. Let us mention that to get the part ``Moreover'' of $iv)$ we apply Theorem \ref{intn} for the curve consisting of all different components of the curves
$\sigma_1,\sigma_2,\sigma_3.$

In the proof of Proposition  later we will need the following
\begin{lem}\label{2sigma}
Assume that the hypotheses of Proposition \ref{hh} hold and assume additionally that
at least one of the following conditions hold:

(a) A nontrivial linear combination of two polynomials from $\{\sigma_1,\sigma_2,\sigma_3\},$ denoted by $s_2,$
vanishes at $B_1$ and $B_2:\ s_2(B_1)=s_2(B_2)=0.$

(b) A nontrivial linear combination of the polynomials $\{\sigma_1,\sigma_2,\sigma_3\},$
denoted by $s_3,$ vanishes at $B_1, B_2,$ and $B_3 \in \ell_{12}:  s_3(B_1)=s_3(B_2)=s_3(B_3)=0,$ and the set $\Xset''':=\Xset\cup \{B_1,B_2,B_3\}$ is $n$-independent;

Then we have that the statement of Proposition \ref{hh} holds.
\end{lem}
\begin{proof} Let us start with (b). In view of Proposition \ref{extcurve} we can extend the set $\Xset'''$ till a maximal $n$-independent set ${\mathcal Y}\subset s_3,$ by adding $d(n,n-1)-(d(n,n-3)+3)-3=1$ node, denoted by $B,$ i.e.,
$\mathcal Y= \Xset'''\cup\{B\}.$

Thus $s_3$ is a maximal curve of degree $n-1$ for the node set $\Yset.$

Then, in view of Lemma \ref{2cor}, we can find a nontrivial linear combination $s$ of $\sigma_1,\sigma_2,\sigma_3,$ such that $s$ differs from $s_3$ and vanishes on $\mathcal X\cup\{B\}.$

Now consider the polynomial $s\ell_{12}\in \Pi_n,$ which vanishes on the node set ${\mathcal Y}.$  By Proposition \ref{maxcor} we conclude that
$$s\ell_{12}=s_3\ell,\ \hbox{where}\ \ell\in\Pi_1.$$
\noindent The line $\ell_{12}$ differs from $\ell,$ since $s$ differs from $s_3.$ Therefore we get that
\begin{equation}\label{n-2}s_3=\ell_{12}q,\ \hbox{where}\ q\in\Pi_{n-2}.\end{equation}

Now, by using $iii),$ we obtain that $q|_{\Xset}=0.$ Hence the statement of Proposition \ref{hh} holds.

(a) Assume, without loss of generality, that $s_2:=c_1\sigma_1+c_2 \sigma_2,\ s_2\neq 0,$ and $s_2(B_1)=s_2(B_2)=0.$

Let us show that there is a node $B_3\in \ell_{12}$ such that $s_2(B_3)\neq 0.$
Indeed, assume conversely that $s_2|_{\ell_{12}}=0.$ Then, by Proposition \ref{maxline}, we obtain that
$$s_2=\ell_{12}q,\ q\in\Pi_{n-2},$$
which finishes the proof in the same way as the relation \eqref{n-2}.

Now, note that $s_2$ is a fundamental polynomial for $B_3\in \Xset''':=\Xset\cup\{B_1,B_2,B_3\}.$ By Lemma \ref{XA} the set
$\Xset'''$ is  is $n$-independent.

Then assume, in view of Lemma \ref{2cor}, that $s$ is a nontrivial linear combination of $s_2$ and $\sigma_3$ such that $s(B_3)=0,$ implying that $s|_{\Xset'''}=0.$
Thus the hypothesis of (b) is satisfied.
\end{proof}

Next, let us continue the proof of Proposition \ref{hh}.

By using Lemma \ref{2cor}, consider a nontrivial linear combination of $\sigma_1,\sigma_2,\sigma_3,$ denoted by $s,$ that vanishes at $B_1$ and $B_2.$ Set $\Xset'':=\Xset\cup\{B_1,B_2\}.$

Denote the set of intersection points of the line $\ell_{12}$ and the curve $s, \deg s=n-1,$ by $\mathcal I:=\ell_{12}\cap s.$ We have that $\#\mathcal I=n-1,$ counting also the multiplicities.
Of course $B_1,B_2\in\mathcal I.$

{\bf Case 1.} First consider the case when one of $B_1,B_2,$ say $B_1,$ is a multiple point of intersection, i.e., $D_{\bf a} s(B_1)=0,$ where ${\bf a}$ is the direction vector of the line $\ell_{12}.$

\noindent Let us prove that the set $\Yset:=\Xset''\cup \{B_1^{(1)}\}=\Xset\cup\{B_1,B_2, B_1^{(1)}\}$ is $n$-independent, where $B_1^{(1)}$ means the directional derivative node with the direction ${\bf a}$ at $B_1.$ According to Lemma \ref{XA} we need to point out a fundamental polynomial $q\in \Pi_n,$ for $B_1^{(1)}\in\Yset,$ i.e., $q|_{\Xset''}=0$ and $D_{\bf a} q(B_1)\neq 0.$

For this end consider a nontrivial polynomial $s_0:=c_1\sigma_1+c_2 \sigma_2,\ s_0\neq 0,$ which vanishes at $B_2:\ s_0(B_2)=0.$

In view of  Lemma \ref{2sigma} we may assume that $s_0(B_1)\neq 0.$

Then consider  a line $\ell$ passing through $B_1$ with a direction vector different from ${\bf a}.$
One can verify readily that the polynomial $q:=\ell s_0$ is a desired polynomial.
Indeed, we have that $q|_{\Xset''}=0.$ Then we have that

$$D_{\bf a} q(B_1)=D_{\bf a} [\ell s_0](B_1)$$ $$=(D_{\bf a} \ell)(B_1)s_0(B_1)+\ell(B_1)D_{\bf a} s_0(B_1)=(D_{\bf a} \ell)(B_1)s_0(B_1)\neq 0.$$

Thus the set $\Yset$ is $n$-independent and hence this case can be proved in the same way as Lemma \ref{2sigma} (b).

{\bf Case 2.} It remains to consider the case when both $B_1$ and $B_2$ are simple points of intersection.
We have that $\#\mathcal I=n-1\ge 3.$
Consider another point of intersection of $\ell_{12}$ and $s:\ B\in \mathcal I,\ B\neq B_1,B_2.$

In view of Lemma \ref{2sigma}, (b), we may assume the following

\noindent {\bf Assumption 1.} \emph{The set $\Xset''\cup\{B\}=\Xset\cup\{B_1,B_2,B\}$ is $n$-dependent.}

This here means that
$\ p\in\Pi_{n},\ p|_{\Xset''}=0 \implies p(B)=0.$

Now consider two nontrivial linear combinations $s_1, s_2$ of $\sigma_1,\sigma_2$ such that
$s_1(B_2)=s_2(B_1)=0.$

By Lemma \ref{2sigma}, (a), we get that $s_i(B_i)\neq 0,\ i=1,2.$
Assume, without loss of generality, that $s_i(B_i)=1,\ i=1,2.$

Next let us show that $s_i(B)=0,\ i=1,2.$ Let say $i=1.$ Consider the polynomial $q:=\ell s_1\in \Pi_n,$ where the line $\ell$ passes through $B_1$ and does not pass through $B.$ We have that $q(B_1)=q(B_2)=0.$ By using Assumption 1 and Lemma \ref {XA} we get that $q(B)=0$ hence $s_1(B)=0.$

Now we are in a position to show that $\sigma_1(B)=\sigma_2(B)=\sigma_3(B)=0.$ 

Let us show for example that $\sigma_1(B)=0.$

Consider the polynomial $p=\sigma_1-c_1s_1-c_2s_2,$ where $c_i=\sigma_1(B_i).$
We get readily that $p(B_1)=p(B_2)=0.$ Hence, in view of Assumption 1, as above, we get that $p(B)=0.$ It remains to note that $\sigma_1(B)=p(B)=0.$

Next suppose that the point $B$ is multiple:

$s(B)=D_{\bf a}s(B)=\ldots,D^{(k)}_{\bf a}s(B)=0,\ k\in\mathbb N.$

In view of Lemma \ref{2sigma}, (b), we may assume the following

\noindent {\bf Assumption 2.} \emph{The set $\Xset''\cup\{B^{(i)}\},\  i=0,\ldots, k$ is $n$-dependent.}

This here means that
\begin{equation} \label{DD}p\in\Pi_{n},\ p|_{\Xset''}=0 \implies p(B)=D_{\bf a}p(B)=\ldots=D^{(k)}_{\bf a}p(B)=0.\end{equation}
Now consider the above defined polynomials  $s_1$ and $s_2$ with
$$s_1(B_1)=1, s_1(B_2)=s_1(B)=0,\ \ s_2(B_2)=1, s_2(B_1)=s_1(B)=0.$$

By using induction on $k$ let us show that
\begin{equation} \label{0k} D^{(i)}_{\bf a}s_j(B)=0,\ i=0,1,\ldots,k,\ j=1,2.
\end{equation}
\noindent Let say $j=1.$ The first step of induction is the above considered case $k=0.$  Assume that the case of $k-1$ is true, i.e., the first $k$ equalities in \eqref{0k} hold. Let us prove the last one, i.e., $D^{(k)}_{\bf a}p(B)=0.$

Consider the polynomial $q:=s_1\ell\in \Pi_n,$ where the line $\ell$ passes through $B_1$ and does not pass through $B.$ We have that $q(B_1)=q(B_2)=0.$ In view of Assumption 2 we get that
$$0=D_{\bf a}^{(k)}[s_1\ell](B)=D^{(k)}_{\bf a}s_1(B)\ell(B)+kD_{\bf a}^{(k-1)}s_1(B)D_{\bf a}\ell(B)=D^{(k)}_{\bf a}s_1(B)\ell(B).$$
Since $\ell(B)\neq0$ we conclude that $D^{(k)}_{\bf a}s_1(B)=0.$

Now we are in a position to show that
\begin{equation} \label{0sigmak} D^{(i)}_{\bf a}\sigma_1(B)= D^{(i)}_{\bf a}\sigma_2(B)= D^{(i)}_{\bf a}\sigma_3(B)=0,\ i=0,1,\ldots,k.\end{equation}
Let us prove say equalities  with $\sigma_1.$
Consider the polynomial
\begin{equation} \label{ssigma}p=\sigma_1-c_1s_1-c_2s_2,\ \hbox{where}\ c_i=\sigma_1(B_i).
\end{equation}
We get readily that $p(B_1)=p(B_2)=0.$ Hence, in view of Assumption 2, as above, we get that $p(B)=D_{\bf a}p(B)=\ldots=D^{(k)}_{\bf a}p(B)=0.$ It remains to use the relations \eqref{0k}  and \eqref{ssigma}.

Hence except the two intersection points $B_1,B_2\in \mathcal I:=\ell_{12}\cap ,s$ all other $n-3$ points, counting also the multiplicities, are common for the three curves $\sigma_1,\sigma_2,$ and $\sigma_3.$

 From this, in view of the condition $(iv)$ (page 5), we conclude that the above three polynomials $\sigma_1,\sigma_2,$ and $\sigma_3,$ have a common divisor $q\in\Pi_{n-3}:$

 $$\sigma_1=\beta_1 q,\quad\sigma_2=\beta_2 q,\quad\sigma_3=\beta_3 q,\ \hbox{where}\ \beta_i\in\Pi_2.$$
  Therefore we have that
\begin{equation} \label{qb}\Xset\subset \sigma_1\cap\sigma_2\cap\sigma_3\subset  q\cup[\beta_1\cap\beta_2\cap\beta_3].\end{equation}

 Now consider two cases for $\Bset:=\beta_1\cap\beta_2\cap\beta_3:$

 {\bf Case (a)},  $\#\Bset\ge 4.$ 
 
 According to Proposition \ref{PnX} any subset $\Aset\subset\Bset$ with $\#\Aset=4$ is $2$-dependent. From here we obtain readily that the points of $\Aset$ are collinear. Hence all the points of $\Bset$ are collinear: $\Bset\subset\ell\in\Pi_1.$

 Now we readily get that $\ell$ is a common divisor of $\beta_1,\beta_2,$ and $\beta_3,$     i.e.,
 $$\beta_1=\ell_1\ell,\quad\beta_2=\ell_2\ell,\quad\beta_3=\ell_3\ell,$$
 where $\ell_i\in\Pi_1.$
Thus, as above, we get that

\begin{equation}\label{Bl}\Bset\subset \ell\cup[\ell_1\cap\ell_2\cap\ell_3]\subset \ell.\end{equation}

The last relation here we get from the fact that the polynomials $\sigma_1, \sigma _2,\sigma_3,$ and hence the polynomials $\ell_1, \ell_2\, \ell_3,$ are linearly independent and hence $\ell_1\cap\ell_2\cap\ell_3=\emptyset.$

Finally, we get from \eqref{qb} and \eqref{Bl} that
$$\Xset\subset q\cup\ell,$$
or, in other words, all the nodes of ${\mathcal X}$  lie in a  curve of degree $n-2,$ namely in the curve $q\ell\in\Pi_{n-2}.$

{\bf Case (b)},  $\#\Bset\le 3.$ 

In this case we obtain from \eqref{qb} that
all the nodes of  ${\mathcal X}$ but $\le 3$ lie in a curve $q$ of degree $n-3.$ From here we readily conclude that $q$ is a maximal curve and exactly $3$ nodes of $\Xset$ are outside of it. 

Thus Proposition \ref{hh} is proved.

Finally note that in view of Theorem \ref{hkv} and Proposition \ref{hh} one can formulate the following
\begin{thm}\label{hkv'}
Assume that ${\mathcal X}$ is an $n$-independent set of $d(n, k-2)+3$ nodes with $3\le k\le n-1.$ Then at most three linearly independent curves of degree
$\le k$ may pass through all the nodes of ${\mathcal X}.$ Moreover, there are such three curves for the set ${\mathcal X}$ if and only if all the nodes of ${\mathcal X}$  lie in a  curve of degree $k-1,$ or all the nodes of  ${\mathcal X}$ but three lie in a (maximal) curve of degree $k-2.$
\end{thm}

\vspace{1cm}

%%
%% Here goes the Bibliography
%%

{\def\section*#1{}
\vskip 5pt
\begin{center}
{\bf REFERENCES}
\end{center}
%\vskip 2pt

The work was carried out under grant 21T-A055 from the Scientific Committee of the Ministry of ESCS RA.

%\end{document}

\end{document}